\documentclass[reqno,11pt]{amsart}
\usepackage{amsmath, latexsym, amsfonts, amssymb, amsthm, amscd}

\usepackage{tikz}
\usepackage{xcolor}
\usepackage{nicematrix}

\usepackage[utf8]{inputenc}
\usepackage{glossaries}

\RequirePackage{fix-cm}
\usepackage{amsmath, latexsym, amssymb}
\usepackage{graphicx}
\usepackage{hyperref}
\usepackage{graphics,epsf,psfrag}
\usepackage{caption}
\usepackage{subcaption}
\numberwithin{equation}{section}

\newtheorem{theorem}{Theorem}[section]
\newtheorem*{thm}{Theorem}
\newtheorem{lemma}[theorem]{Lemma}
\newtheorem*{lem}{Lemma}
\newtheorem{proposition}[theorem]{Proposition}
\newtheorem*{prop}{Proposition}
\newtheorem{cor}[theorem]{Corollary}
\newtheorem{rem}[theorem]{Remark}

\newcommand{\R}{\mathbb{R}}
\newcommand{\Z}{\mathbb{Z}}
\newcommand{\N}{\mathbb{N}}
\renewcommand{\tilde}{\widetilde}

\newcommand{\bP}{{\ensuremath{\mathbf P}} }
\newcommand{\bE}{{\ensuremath{\mathbf E}} }

\DeclareMathSymbol{\leqslant}{\mathalpha}{AMSa}{"36} 
\DeclareMathSymbol{\geqslant}{\mathalpha}{AMSa}{"3E} 
\DeclareMathSymbol{\eset}{\mathalpha}{AMSb}{"3F}     
\renewcommand{\leq}{\;\leqslant\;}                   
\renewcommand{\geq}{\;\geqslant\;}                   
\newcommand{\dd}{\text{\rm d}}             

\newcommand{\bbC}{{\ensuremath{\mathbb C}} }

\newcommand{\bbR}{{\ensuremath{\mathbb R}} }

\newcommand{\norm}[1]{\left\lVert#1\right\rVert}

\title{Wick theorem for analytic functions of Gaussian fields}

\author[F. Coppini] {\small Fabio Coppini } 
\address{Mathematical Institute, Utrecht University, Budapestlaan 6, 3584 CD Utrecht, The Netherlands}
\vspace{-1cm}
\email{f.coppini@uu.nl}
\author[W. M. Ruszel]{\small Wioletta M. Ruszel}
\address{Mathematical Institute, Utrecht University, Budapestlaan 6, 3584 CD Utrecht, The Netherlands}
\email{w.m.ruszel@uu.nl}
\author[D. Schuricht]{\small Dirk Schuricht}
\address{Institute for Theoretical Physics Utrecht University, Princetonplein 5, 3584 CC Utrecht, The Netherlands}
\email{d.schuricht@uu.nl}
\date{\today}

\begin{document}

\begin{abstract}
We compute the correlation of analytic functions of general Gaussian fields in terms of multigraphs and Feynman diagrams on the lattice $\mathbb{Z}^d$. Then, we connect its scaling limit to tensors of the correlation functionals of Fock space fields. Afterwards, we investigate the relation with fermionic Gaussian field states for even functions. For instance, we characterize the correlation functionals of the exponential of a continuous Gaussian Free Field or general analytic functions of fractional Gaussian fields as limits of quantities constructed via a sequence of discrete fields. Finally, we show that the duality between even powers of bosonic Gaussian fields and "complex" fermionic Gaussian fields can be reformulated in terms of a principal minors assignment problem of the corresponding covariance matrices.\\

\noindent   {\it MSC2020}  {\it Subject classifications:} 60G15, 60G22, 60G60, 60K35, 81T18, 81V73, 81V74, 81Q60\\
    

\end{abstract}

\maketitle

\section{Introduction}

\subsection{Aim of this work and related literature}

The Gaussian Free Field (GFF) $\phi$ (see,  e.g., \cite{sheffield_gaussian_2007}) is a fundamental object in probability
theory and mathematical physics. It is the canonical model for a random surface or
generalized function-valued Gaussian process, and the prototypical non-interacting
(quantum) field. In particular, the GFF arises as a universal scaling limit for fluctuations
in a wide class of lattice and random interface models, making it central in the study
of critical phenomena in statistical mechanics. Most physically relevant quantum field theories, however, involve interactions,
typically modeled by nonlinear functions of the field such as $e^{\alpha\phi}$,
$\phi^4$, or $\cos(\beta\phi)$. These analytic functions appear naturally in the
action functionals of models including Liouville quantum gravity~\cite{Berestycki_Powell_2025}, $\phi^4$
field theory~\cite{hugophi4}, and the sine-Gordon model~\cite{since-G}. Further examples where analytic
functions of the GFF play an important role include Gaussian multiplicative
chaos~\cite{janson_gaussian_1997,peccati_wiener_2011}  and the bosonization correspondence between one-dimensional
fermionic and bosonic quantum field theories~\cite{vonDelftSchoeller98}. In these contexts, analytic
functionals of Gaussian fields encode physically meaningful observables and
correlations.

Wick’s theorem and Isserlis’ theorem~\cite{wick, janson_gaussian_1997, peccati_wiener_2011, yves} provide systematic methods for computing expectations of products of Gaussian variables, e.g., expectations of nonlinear observables in constructive quantum field theory~\cite{GlimmJaffe87}. These results play a central role because correlation functions of interacting field theories are typically expressed
as expectations of products of Gaussian fields arising from perturbative expansions.
Such expansions form the basis of the diagrammatic (Feynman) techniques widely used
in quantum field theory and statistical mechanics, where Wick’s theorem allows one
to reduce products of fields to sums over pairings corresponding to Feynman diagrams
(see e.g.\cite{PeskinSchroeder1995, ZinnJustin02, GlimmJaffe,Simon}). From a combinatorial perspective, these pairings naturally give
rise to graph structures describing the contraction patterns between field insertions.
From a probabilistic viewpoint, Isserlis’ theorem provides a closed-form expression
for expectations of products of jointly Gaussian variables.

While originally formulated in the context of bosonic fields, which are typically positively correlated, Wick’s theorem also applies to fermionic fields, which are typically negatively correlated, with appropriate modifications to account for their anti-commuting nature.
Despite the differences in statistics and algebraic rules (commutators for bosons, anti-commutators for fermions), the combinatorial structure of Wick’s theorem is strikingly similar in both cases, since both reduce higher-order correlation functions to sums over pairings (contractions), they rely on the Gaussian nature of the free field and involve normal ordering. In fact, in \cite{cipriani_properties_2023} the authors compute for the first time the correlations and joint cumulants of the square of the gradient of the GFF on the lattice $\mathbb{Z}^d$, together with the scaling limits of its cumulants. Among other results, in \cite{chiarini_fermion23, lectureAle} the cumulants of the gradient of pairs of fermions are computed (on the lattice and its scaling limits) and it is noted that they are the same up to a negative constant, hinting at an underlying \emph{supersymmetric} relation in the sense of \cite{renorma, lectureAle}. 

The main goal of this work is to develop a systematic framework for computing
correlations of analytic functions of general Gaussian fields on the lattice and
to study their scaling limits. Building on the combinatorial structure underlying
Wick contractions, we express these correlations in terms of multigraph structures
that encode the possible contraction patterns between field insertions.
Our results provide a
framework that connects analytic observables of Gaussian fields with Feynman-type
diagrammatics while extending these techniques to more general nonlinear
observables. In addition, we relate these structures to correlation functionals of
bosonic Fock space fields and establish connections with fermionic Gaussian states
for even observables.

\subsection{Informal statements of the main results}

We will prove that,

\begin{thm}[see Theorem \ref{thm:wick-analytic}]
For a general Gaussian field with covariance matrix $G$ defined on $\Lambda \subset \mathbb{Z}^d$, $d\geq 2$, and $f(x) = \sum_{n} a_n x^n$, we have that for $N$ points $x_1, \dots, x_N$ in $\Lambda$, it holds
\[
\mathbf{E}\left(\prod_{i=1}^N :f(\phi(x_i)):\right) = \sum_{n_1,\ldots,n_N} c_{f, N} \sum_{q \text{ multigraph }} c_q \prod_{i,j=1}^N G(i,j)^{q_{ij}} 
\]
where $c_{f,N}=\prod_{i=1}^{N} a_{n_i}$ and $c_q=const(q)$ is an explicit constant. $q$ is an undirected multigraph with no self-loops.
\end{thm}

We have a similar formula for the cumulants as well, see the full statement in Theorem \ref{thm:wick-analytic}. Moreover, we are able to compute the scaling limit of the correlations of analytic functions of general Gaussian fields. 

Furthermore,  we express analytic functions of general Gaussian fields (including the exponential of a GFF or general analytic functions of fractional Gaussian fields) as Fock space fields and prove that
the scaling limit of the correlations on the discrete level can be expressed in terms of the tensor of the correlation functionals as Fock space fields:

\begin{prop}[see Propositions \ref{pro:non-triv-cumulants} and \ref{pro:non-triv-cumulants-fock}]
Let $\phi$ be a  Gaussian field on $U_\varepsilon=\varepsilon^{-1} U \cap \Z^d$ for which the covariance has a continuous limit 
\[
\lim_{\varepsilon \to 0} \eta(\varepsilon) \, G_\varepsilon\Bigl(x^{(\varepsilon)}, y^{(\varepsilon)}\Bigr) = g_U(x,y),
\]
and let $\overline{\phi}$ be a continuous Gaussian field with covariance operator given by $g_U$. Let $k\geq 2$ be a positive integer, for $j=1, \dots, N$ and $x^{(1)},\dots,x^{(N)}$ disjoint points in $U$ with $x^{(i)}_\varepsilon$ denoting the point in $U_\varepsilon$ the closest to $x^{(i)}/\varepsilon$. For $f$ analytic function, it holds that
\begin{equation}
\lim_{\varepsilon \rightarrow 0} \eta(\varepsilon)^N \bE \left[ \prod_{j=1}^N :f \left(\phi(x^{(j)}_\varepsilon) \right): \right] = \mathcal{Y}_1(f) \bullet  \cdots \bullet \mathcal{Y}_N (f),
\end{equation}
where $\mathcal{Y}_j(f)$ are suitable Fock space fields and $\bullet$ denotes a suitable correlation functional tensor product.
\end{prop}

Finally, we investigate whether the cumulants of general analytic functions of Gaussian fields can be expressed up to negative constants in terms of cumulants with respect to a fermionic Gaussian field which is true for the square of the gradient of the discrete GFF, see \cite{chiarini_fermion23, cipriani_properties_2023}. We generalize these results to any Gaussian field and prove in an alternative elegant way that for the cumulant function $k$, which can be seen as {\it Boson-Fermion cumulant correspondence:}
\begin{thm}[see Theorem \ref{thm:cumulants-p=2}]
Let $G$ be the covariance matrix of a real Gaussian vector $X$, we have that
\begin{equation}\label{eq:id_cumu}
    k\left (\left (: \prod_{i=1}^N X_i:\right )^2 \right ) = 2^{N-2} k \left(: \prod_{i=1}^N X_i^2: \right ) = k\left(\prod_{i=1}^N |Z_i|^2 \right ) = (-1)^{N-1} k_{ferm} \left (\prod_{i=1}^N \psi_i \bar{\psi}_i \right ),
\end{equation}
where $Z$ is a complex Gaussian vector with covariance matrix $G$ and $k_{ferm}$ are the cumulants of the fermions $\psi_1\bar{\psi}_1,\ldots, \psi_N\bar{\psi}_N$ w.r.t. the fermionic Gaussian field state with covariance matrix $G$.
\end{thm}

Note that we can define a general fermionic Gaussian state w.r.t. an arbitrary covariance matrix $C$ as follows: let $|\Lambda|=n$ and $C$ be a $nr\times nr$ matrix symmetric, invertible matrix. To each site $i\in \Lambda$ we assign $r$-many pairs of fermions $\prod_{l=1}^r\psi^{(l)}_i \bar{\psi}^{(l)}_i$.
The case $r=1$ correspond to the square case treated in \cite{chiarini_fermion23}. At least for even functions corresponding to analytic expressions in terms of even powers one could expect that it is possible to get similar expressions for the cumulants hinting at an underlying supersymmetric structure.
We will show that the positive answer of this claim, i.e. an identity like Equation \eqref{eq:id_cumu} for general powers, will depend on determining the principal minors of the inverse of the covariance matrix: 

\begin{lem}[see Lemma \ref{lem:bos-fer-p>2}]
    Let $A=\{1,\ldots,N\}$. The following conditions are equivalent
    \begin{enumerate}
        \item The fermionic joint cumulants satisfy
        \begin{equation}
            k \left (\prod_{l=1}^r \psi_i^{(l)}\bar{\psi}^{(l)}_i; i=1,\ldots, N \right) = (-1)^{|A|-1} k \left( \prod_{i=1}^k |Z_i|^{2r}\right).
        \end{equation}
        \item Let $A_r=\cup_{i\in A} \{ (i-1)r, (i-1)r+1,\dots,ir\}$
        \begin{equation}
            (\det(C_{A_r, A_r}))^{-1} = \sum_{q \text{ multigraph}} C_q \prod_{i, j=1}^N G(i,j)^{q_{ij}},
        \end{equation}
        where $C_q$ is some constant depending on $q$ and $C_{A_{r}, A_r}$ is the matrix $C$ restricted to rows and columns depending on $A_r$.
    \end{enumerate}
\end{lem}
Note that this is a long-standing open problem in algebra, see \cite{holtz_hyperdeterminantal_2007, huang_symmetrization_2017, matoui_principal_2021}.

\subsection{Organization of the paper}

In Section \ref{s:wick} we will first introduce the Wick product, Feynman diagrams and multigraphs and the main theorem regarding computing the correlation functions resp. cumulants of analytic functions of general Gaussian fields in terms of multigraphs on the lattice $\mathbb{Z}^d$. We end the section with a useful result about a expressions for cumulant functions for general functions. 

In Section \ref{s:fock-space} we introduce Fock spaces and define analytic functions of general Gaussian fields in the continuum as Fock space fields. We will match scaling limits of correlations of analytic functions of Gaussian fields to Fock space correlation functionals of the corresponding continuum fields. We will also discuss some concrete examples such as the exponential function of a GFF.

Finally in Section \ref{s:bosonic-fermionic} we will first introduce Grassmann-Berezin calculus and introduce general fermionic Gaussian field states. First, using the cumulant identity proven in Section \ref{s:wick} we will be able to establish a duality (up to constants) for cumulants relating fermionic and bosonic (real and complex) pairs to each other. For general even powers of Gaussian fields we prove that one can establish that the cumulants of the bosonic and fermionic parts are equal up to constants if and only if the principal minors assignment problem of the determinant of a certain matrix can be solved. 

\section{Wick theorem for analytic functions of Gaussian fields}
\label{s:wick}

\subsection{The Wick product}

The notation and most of the auxiliary results of this section are taken from \cite{janson_gaussian_1997}. Let $(\Omega, \mathcal{F}, \bP)$ be a probability space. Fix $d\geq 1$ and denote by $H$ be a Gaussian Hilbert space on $\R^d$ (linear subspace of $L^2(\Omega, \mathcal{F}, \bP)$ which is complete and consists of centered Gaussian random variables). For $n$ positive integer, let $\bar{\mathcal{P}}_n(H)$ be the closure in $L^2(\Omega, \mathcal{F}, \bP)$ of the linear space
\begin{equation}
    \mathcal{P}_n(H) \overset{\text{def}}= \{p(\xi_1, \dots, \xi_m) \text{ polynomial of degree} \leq n : \; \xi_1, \dots, \xi_m \in H, m< \infty\}
\end{equation}
and let
\begin{equation}
    H^{:n:} \overset{\text{def}}= \bar{\mathcal{P}}_n(H) \cap\bar{\mathcal{P}}_{n-1}(H) = \bar{\mathcal{P}}_n(H) \cap \bar{\mathcal{P}}_n(H)^{\perp}, \quad H^{:0:} \overset{\text{def}} = \R.
\end{equation}

It is easy to see that $\bar{\mathcal{P}}_n(H) = \bigoplus_{k=0}^n H^{:k:}$ and that $\bigoplus_{k=0}^\infty H^{:k:} = L^2(\Omega, \mathcal{F}, \bP)$. When the context is clear, we will use $L^2$ instead of $L^2(\Omega, \mathcal{F}, \bP)$. Let $\pi_n: L^2\to H^{:n:}$ be the standard orthogonal projection. Clearly, any $X\in L^2$ can be written as $X=\sum_{k=0}^\infty \pi_k(X)$. 

The \emph{Wick product} of $\xi_1, \dots, \xi_n\in H$ is defined by
\begin{equation}
\label{def:classical-wick-product}
    :\xi_1 \cdots \xi_n: \; \; \overset{\text{def}}= \pi_n (\xi_1 \cdots \xi_n), \quad :: \;  \overset{\text{def}}= 1\in H^{:0:}.
\end{equation}
The definition can be extended to $X,Y\in L^2$: let $X\in \bar{\mathcal{P}}_n(H)$ and $Y\in \bar{\mathcal{P}}_m(H)$ for some $n,m\geq 0$, the {\it Wick product} of $X$ and $Y$ is defined by
\begin{equation}
\label{def:fock-wick-product}
    X\odot Y \overset{\text{def}}= \pi_{n+m}(XY).
\end{equation}
A classical result on Gaussian Hilbert spaces ensures that if $X,Y\in L^2$, then also their product $XY$ lies in $L^2$, this means that the Wick product defined above is a well defined object. While the notation in \eqref{def:classical-wick-product} is standard, the one in \eqref{def:fock-wick-product} is motivated by the relation with Fock spaces that we recall in Section \ref{s:fock-space}.

From the definition of $\pi_n$, it is easy to see that for $\xi_1, \xi_2, \xi_3$ Gaussian variables in $H$ it holds that
\begin{equation}
    \begin{split}
        &:\xi_1: \; = \xi - \bE[\xi]\\
        &:\xi_1 \xi_2: \; = \xi_1 \xi_2 - \bE[\xi_1\xi_2]\\
        &:\xi_1 \xi_2 \xi_3: \; = \xi_1\xi_2\xi_3 - \bE[\xi_1\xi_2]\xi_3 - \bE[\xi_1\xi_3]\xi_2 - \bE[\xi_2\xi_3] \xi_1.
    \end{split}
\end{equation}

\subsection{Feynman diagrams and multigraphs}

Feynman diagrams are a powerful graphical tool, originally developed in quantum field theory, to visualize and compute combinations of Wick products. As we will only work with complete diagrams, we restrict to this subset of diagrams. 

\begin{figure}[h!]
\centering
\begin{tikzpicture}[scale=1.2]
    \node[circle,draw,fill=blue!20,label=above:$X_1$] (A1) at (-1,1) {};
    \node[circle,draw,fill=blue!20,label=above:$\cdot$] (A2) at (-0.5,1) {};
    \node[circle,draw,fill=blue!20,label=above:$\cdot$] (A3) at (0,1) {};
    \node[circle,draw,fill=blue!20,label=above:$X_1$] (A4) at (0.5,1) {};
    \node[circle,draw,fill=red!20,label=below:$X_4$] (B1) at (-1,-1) {};
    \node[circle,draw,fill=red!20,label=below:$\cdot$] (B2) at (-0.5,-1) {};
    \node[circle,draw,fill=red!20,label=below:$\cdot$] (B3) at (0,-1) {};
    \node[circle,draw,fill=red!20,label=below:$X_4$] (B4) at (0.5,-1) {};
    \node[circle,draw,fill=green!20,label=above:$X_2$] (C1) at (2,1) {};
    \node[circle,draw,fill=green!20,label=above:$\cdot$] (C2) at (2.5,1) {};
    \node[circle,draw,fill=green!20,label=above:$\cdot$] (C3) at (3,1) {};
    \node[circle,draw,fill=green!20,label=above:$X_2$] (C4) at (3.5,1) {};
    \node[circle,draw,fill=purple!20,label=below:$X_3$] (D1) at (2,-1) {};
    \node[circle,draw,fill=purple!20,label=below:$\cdot$] (D2) at (2.5,-1) {};
    \node[circle,draw,fill=purple!20,label=below:$\cdot$] (D3) at (3,-1) {};
    \node[circle,draw,fill=purple!20,label=below:$X_3$] (D4) at (3.5,-1) {};
    
    \draw (A4) to[out=-45,in=-135] (C1);
    \draw (A3) to[out=-45,in=-135] (C2);
    \draw (B4) to[out=45,in=135] (D1);
    \draw (B3) to[out=45,in=135] (D2);
    \draw (A1) to[out=-90,in=90] (B1);
    \draw (A2) to[out=-90,in=90] (B2);
    \draw (C3) to[out=-90,in=90] (D3);
    \draw (C4) to[out=-90,in=90] (D4);

    \begin{scope}[shift={(6,0)}]
    \node[circle,draw,fill=blue!20,label=above:$X_1$] (A1) at (-1,1) {};
    \node[circle,draw,fill=blue!20,label=above:$\cdot$] (A2) at (-0.5,1) {};
    \node[circle,draw,fill=blue!20,label=above:$\cdot$] (A3) at (0,1) {};
    \node[circle,draw,fill=blue!20,label=above:$X_1$] (A4) at (0.5,1) {};
    \node[circle,draw,fill=red!20,label=below:$X_4$] (B1) at (-1,-1) {};
    \node[circle,draw,fill=red!20,label=below:$\cdot$] (B2) at (-0.5,-1) {};
    \node[circle,draw,fill=red!20,label=below:$\cdot$] (B3) at (0,-1) {};
    \node[circle,draw,fill=red!20,label=below:$X_4$] (B4) at (0.5,-1) {};
    \node[circle,draw,fill=green!20,label=above:$X_2$] (C1) at (2,1) {};
    \node[circle,draw,fill=green!20,label=above:$\cdot$] (C2) at (2.5,1) {};
    \node[circle,draw,fill=green!20,label=above:$\cdot$] (C3) at (3,1) {};
    \node[circle,draw,fill=green!20,label=above:$X_2$] (C4) at (3.5,1) {};
    \node[circle,draw,fill=purple!20,label=below:$X_3$] (D1) at (2,-1) {};
    \node[circle,draw,fill=purple!20,label=below:$\cdot$] (D2) at (2.5,-1) {};
    \node[circle,draw,fill=purple!20,label=below:$\cdot$] (D3) at (3,-1) {};
    \node[circle,draw,fill=purple!20,label=below:$X_3$] (D4) at (3.5,-1) {};
    
    \draw (A4) to[out=-45,in=-135] (C1);
    \draw (A3) to[out=-45,in=-135] (C2);
    \draw (B4) to[out=45,in=135] (D2);
    \draw (B3) to[out=45,in=135] (D1);
    \draw (A1) to[out=-90,in=90] (B2);
    \draw (A2) to[out=-90,in=90] (B1);
    \draw (C3) to[out=-90,in=90] (D3);
    \draw (C4) to[out=-90,in=90] (D4);
    \end{scope}

\node[below] at (1, -2) {Feynman diagram $\gamma_1$};
\node[below] at (7, -2) {Feynman diagram $\gamma_2$};

\end{tikzpicture}
\caption{Two different Feynman diagrams, labeled by the centered Gaussian random variables $X_1, \dots, X_4$. Although different, $\gamma_1$ and $\gamma_2$ have the same value, i.e, $\nu(\gamma_1) = \nu(\gamma_2) = \prod_{1\leq i<j\leq 4} \bE\left[X_i X_j\right]^2$.}
\label{fig:feynman-couple} 
\end{figure}

A \emph{complete Feynman diagram $\gamma$} of order $2r$ is a graph consisting of a set of $2r$ vertices and $r$ edges without common endpoints, i.e., a set of $r$ distinct 2-tuples. Let $FD_0$ denote the set of all complete Feynman diagrams. A Feynman diagram $\gamma \in FD_0$ \emph{labeled} by $2r$ random variables $X_1, \dots, X_{2r}$, defined on the same probability space, is a Feynman diagram of order $2r$ where each variable $X_i$ is associated to the vertex $i$. The \emph{value $\nu(\gamma)$} of such diagram is given by
\begin{equation}
    \nu(\gamma) \overset{\text{def}}= \prod_{e\in E_\gamma} \bE \left[X_{e^-} X_{e^+}\right],
\end{equation}
where $E_\gamma$ denotes the set of edges of $\gamma$ and $e=(e^-, e^+)$ with $e^-, e^+ \in [2r]$ and $[n]=\{1,\ldots,n\}$.

A well known result about expectations of products of Wick products is given by the following theorem.
\begin{theorem}[{\cite[Theorem 3.12]{janson_gaussian_1997}}]
\label{thm:janson-feynman}
    Let $Y_i= :\xi_{i1} \cdots \xi_{il_i}:$, where $( \xi_{ij})_{\substack{1\leq i \leq k\\1\leq j \leq l_i}}$ and $k, l_1, \dots, l_k\geq 0$ are centered jointly Gaussian variables. Then,
    \begin{equation}
        \bE\left[\prod_{i=1}^k Y_i\right] = \sum_{\gamma} \nu(\gamma), 
    \end{equation}
    where the sum ranges over all complete Feynman diagrams $\gamma$ labeled by $\xi_{ij}$ such that no edge joins two variables $\xi_{i_1j_1}$ and $\xi_{i_2j_2}$ with $i_1 = i_2$.
\end{theorem}

If for each $i=1, \dots, k$ the variable $Y_i= \;:\xi_{i}^{l_i}:$, then the previous theorem can be rewritten in terms of multigraphs without self-loop (this is partially formulated in \cite[Remark 3.13]{janson_gaussian_1997}). Given $B$ a finite subset of $\N$, and $(l_i)_{i\in B}$ such that $l_i\in \N$ for every $i\in B$, we denote by $MG_0(B,(l_i)_{i\in B})$ (or simply $MG_0$ when the context is clear) \textit{the set of undirected multigraphs with vertices labeled by $B$, no self-loop and exactly $l_i$ edges per vertex $i$}. Observe that, depending on $|B|$ and $(l_i)_{i\in B}$, the set $MG_0$ can be empty.

Every $q\in MG_0$ is uniquely characterized by a $|B| \times |B|$ symmetric matrix $q = (q_{ij})_{i,j\in B} \in \R^{|B|} \times \R^{|B|}$ with null diagonal, taking values in $\{0, 1, \dots, \max_{i\in B}{l_i}\}$ and such that
\begin{align}
    & \sum_{j,l\in B} q_{jl} =  \sum_{i\in B} l_i.
\end{align}
In the sequel, a \textit{connected multigraph} is a multigraph such that there is a path between any two vertices. See Figure \ref{fig:multigraph} for an example of connected multigraphs.

\begin{figure}[h!]
\centering
\begin{tikzpicture}[scale=2]

\node[above right] at (-3, 1.5) {$B=\{1, \dots, 4\}, \; l_i\equiv 4 \quad q, \tilde{q} \in MG_0$};

\begin{scope}[shift={(-1.5,0)}]
\node[circle, draw, fill=blue!20, minimum size=1cm] (E1) at (0, 1) {$e^{(1)}$};
\node[circle, draw, fill=green!20, minimum size=1cm] (E2) at (1, 0) {$e^{(2)}$};
\node[circle, draw, fill=purple!20, minimum size=1cm] (E3) at (0, -1) {$e^{(3)}$};
\node[circle, draw, fill=red!20, minimum size=1cm] (E4) at (-1, 0) {$e^{(4)}$};

\foreach \angle in {135, 45, -45, -135} {
  \fill (E1.\angle) circle (1.2pt);
  \fill (E2.\angle) circle (1.2pt);
  \fill (E3.\angle) circle (1.2pt);
  \fill (E4.\angle) circle (1.2pt);
}

\draw (E1) to[out=135,in=135] (E4);

\draw (E2) to[out=315,in=315] (E3);

\draw (E1) to[out=45,in=45] (E2);
\draw (E1) to[out=315,in=135] (E2);
\draw (E1) to[out=225,in=225] (E2);

\draw (E4) to[out=225,in=225] (E3);
\draw (E4) to[out=315,in=135] (E3);
\draw (E4) to[out=45,in=45] (E3);





\node[below] at (0, -1.5) {$q$ such that $q_{23} = 1 = q_{14}$};
\end{scope}

\begin{scope}[shift={(2,0)}]
    \node[circle, draw, fill=blue!20, minimum size=1cm] (E1) at (0, 1) {$e^{(1)}$};
    \node[circle, draw, fill=green!20, minimum size=1cm] (E2) at (1, 0) {$e^{(2)}$};
    \node[circle, draw, fill=purple!20, minimum size=1cm] (E3) at (0, -1) {$e^{(3)}$};
    \node[circle, draw, fill=red!20, minimum size=1cm] (E4) at (-1, 0) {$e^{(4)}$};
    
    \foreach \angle in {135, 45, -45, -135} {
      \fill (E1.\angle) circle (1.2pt);
      \fill (E2.\angle) circle (1.2pt);
      \fill (E3.\angle) circle (1.2pt);
      \fill (E4.\angle) circle (1.2pt);
    }
    
    \draw (E1) to[out=135,in=135] (E4);
    \draw (E1) to[out=225,in=45] (E4);
    
    \draw (E2) to[out=225,in=45] (E3);
    \draw (E2) to[out=315,in=315] (E3);
    
    \draw (E1) to[out=45,in=45] (E2);
    \draw (E1) to[out=315,in=135] (E2);
    
    \draw (E4) to[out=225,in=225] (E3);
    \draw (E4) to[out=315,in=135] (E3);

\node[below] at (0, -1.5) {$\tilde{q}$ such that $\tilde{q}_{23} = 2 = \tilde{q}_{12}$};
\end{scope}

\end{tikzpicture}
\caption{Examples of two distinct connected multigraphs on 4 vertices and 4 edges per vertex. Each vertex has $4$ nodes that shall not be connected with each others, i.e., no self-loop. These nodes are exchangeable at the level of the graph, i.e., if you rewire an edge from one node to another in the same vertex, the multigraph remains the same. However, they do play a role in the computation of Gaussian correlation functions as exchanging them yield a different Feynman diagram. See the proof of Theorem \ref{thm:wick-analytic} for more on that.}
\label{fig:multigraph}
\end{figure}

We state and prove a refinement of Theorem \ref{thm:janson-feynman} when the Wick product is applied to powers of the same centered Gaussian variable. Observe that for this case one also has a characterization given by Hermite polynomials, see \cite[Theorem 3.21]{janson_gaussian_1997}.

\begin{theorem}
\label{thm:wick-graph}
    Let $Y_i= \; :\xi_i^{l_i}:$, where $( \xi_i)_{1\leq i \leq k}$ are centered jointly Gaussian variables, with $k, l_1, \dots, l_k$ positive integers. Then,
    \begin{equation}
        \bE\left[\prod_{i=1}^k Y_i\right] = \sum_{q \in MG_0([k], (l_i)_i)} \delta(q)\prod_{1\leq i \leq j\leq k}\bE[\xi_i \xi_j]^{q_{ij}}, 
    \end{equation}
    where 
    \begin{equation}
    \label{def:delta-q}
        \delta(q)=\prod_{1\le i \leq j\le k} \binom{l_i}{q_{ij}}\,\binom{l_j}{q_{ij}}\,q_{ij}!.
    \end{equation}
\end{theorem}

\begin{proof}
The characterization via Feynman diagrams given in Theorem \ref{thm:janson-feynman} yields that
\begin{equation}
    \bE\left[ \prod_{j=1}^k Y_i \right] =  \bE \left[ \prod_{j=1}^k : \xi^{l_i}_i: \right]= \sum_{\gamma \in FD_0} \prod_{b\in E_\gamma} \bE\left[ \xi_{b^+} \xi_{b^-} \right],
\end{equation}
where $FD_0$ denotes the set of all complete Feynman diagrams on the space 
\[
\mathcal{U} = \{\underbrace{\xi_1, \ldots, \xi_1}_{l_1 \text{ times }}, \underbrace{\xi_2, \ldots, \xi_2}_{l_2 \text{ times }}, \ldots, \underbrace{\xi_k, \ldots, \xi_k}_{l_k \text{ times }}\},
\]
where there are $l_i$ copies of every $\xi_i$ for $i=1, \dots, k$. The edges of $\gamma$ are denoted by $E_\gamma$, with the notation that $b = (b^+, b^-) \in E_\gamma$ is such that $\xi_{b^+}, \xi_{b^-} \in \mathcal{U}$.

We need to enumerate all complete Feynman diagrams on $\mathcal{U}$ so that a couple cannot be formed by two copies of some $\xi_i$. Each Feynman diagram corresponds to an undirected multigraph on $k$ vertices, where the $i$-th vertex has exactly $l_i$ edges and no self-loop for every $i=1, \dots, k$. Vice versa, an undirected multigraph may correspond to multiple Feynman diagrams and we denote this multiplicity number by $\delta(q)$, see Figure \ref{fig:feynman-couple} for an example. We obtain that
\begin{equation}
\begin{split}
    \bE\left[ \prod_{j=1}^k Y_i \right] =&\sum_{\gamma \in FD_0} \prod_{b\in E_\gamma} \bE\left[ \xi_{b^+} \xi_{b^-} \right]\\
    =& \sum_{q\in MG_0} \delta(q) \prod_{1\leq i\leq j\leq k} \bE[\xi_i \xi_j]^{q_{ij}}.
\end{split}
\end{equation}
Observe that, depending on $|B|$ and $l_i$, the set $MG_0$ might be empty. For instance, if $\sum_{i=1}^k l_i$ is odd, there is no complete Feynman diagram and the expectation is thus zero.

Let's now compute $\delta(q)$:
for a fixed multigraph $q\in MG_0(k,(l_i)_i)$, we wish to count the number $\delta(q)$ of Feynman diagrams which give rise to $q$ when we “merge” the $l_i$ copies of $\xi_{i}$ for each $i$. Let us focus on a pair of distinct vertices $i$ and $j$ (with $i<j$ because we do not consider self-loops): we think vertex $i$ having $l_i$ half–edges and vertex $j$  with $l_j$ half–edges. In any pairing that gives rise to the edge count $q_{ij}$ between vertices $i$ and $j$, one must:
\begin{enumerate}
    \item Choose which half–edges get paired between $i$ (and equivalently for $j$): from the $l_i$ copies at vertex $i$, choose $q_{ij}$ of them. This can be done in $\binom{l_i}{q_{ij}}$ ways.
    \item Pair the chosen half–edges: Once the two sets (each of size $q_{ij}$) have been chosen, there are $q_{ij}!$ ways to pair the $q_{ij}$ half–edges from vertex $i$ with those from vertex $j$.
\end{enumerate}
Since the pairings for different pairs $(i,j)$ are made independently, we obtain
\[
\delta(q)=\prod_{1\le i<j\le k} \binom{l_i}{q_{ij}}\,\binom{l_j}{q_{ij}}\,q_{ij}!\,. 
\]
The proof is concluded.
\end{proof}

When $Y_i$ are complex Gaussian variables, Theorem \ref{thm:wick-graph}, can be reformulated in terms of the permanent of a suitable covariance matrix, we refer to Equation \eqref{eq:complex-wick-theorem} for the precise statement.
We finish this section with a useful remark linking permutations with multigraphs with exactly $2$ edges per site. 

\begin{rem}
\label{rem:duality-permutation-mg}
A permutation $\sigma$ can be represented with a $\{0,1\}$-valued $|B|\times|B|$ matrix $q_\sigma$, yet not symmetric, such that $q_{jl}=1$ if and only if $l=\sigma(j)$. By symmetrizing the matrix $q^*_\sigma := q_\sigma + q^t_\sigma$, we obtain that $q^*_\sigma \in MG(B, 2)$. Vice versa, any $q\in MG(B, 2)$ identifies a permutation $\sigma_q$ where $\sigma(i)=j$ if and only if $q_{ij}=1$ for $i\leq j$. It is easy to see that 
\[
\prod_{i\in A} g(i, \sigma(i)) = \prod_{i,j\in A, i\leq j} g(i,j)^{q_{ij}}
\]
and that cyclic permutations are associated with connected multigraphs. Similarly, permutations with no fixed point are associated with multigraphs with no self-loop.
\end{rem}

\subsection{Main result}
Before giving the main result of this section, we define Gaussian fields on the lattice. Let $\Lambda \subset \Z^d$ be a finite subset. A {\it Gaussian field} on $\Lambda$ is a random field whose probability measure is given by  
\begin{equation}
\label{def:discrete-gaussian-field}
    \mu_{\Lambda, G} (\dd x) = \frac{1}{Z_{\Lambda, G}} \exp \left( -\frac{1}{2} \langle x, G^{-1} x\rangle \right) \dd x,
\end{equation}
where $G$ is a symmetric positive-definite matrix indexed by $\Lambda \times \Lambda$, and $G^{-1}$ is its inverse, representing the covariance structure of the field. The measure $\dd x$ is the Lebesgue measure on $\R^\Lambda$, defined as $\dd x := \prod_{i\in \Lambda} \dd x_i$. The normalization constant $Z_{\Lambda, G}$ ensures that $\mu_{\Lambda, G}$ is a probability measure.

\subsubsection{Example: the discrete Gaussian free field}
When $G=G^\textup{green}$ is the Green's function of the Laplacian $\Delta_\Lambda$ on $\Lambda$, we obtain the standard discrete Gaussian free field (DGFF) with free boundary conditions. One can recover the DGFF with Dirichlet boundary conditions in the following way. Let $\partial^{\textup{ex}} \Lambda$ denote the outer boundary of $\Lambda$, i.e., $\partial^{\textup{ex}} \Lambda := \{ x\in \Z^d\setminus \Lambda : \exists y \in \Lambda \; y\sim x\}$, and let $\tilde{\Lambda} = \Lambda \cup \partial^{\textup{ex}} \Lambda$. If $\Delta_{\tilde{\Lambda}}$ denotes the (normalized) Laplacian matrix on $\tilde{\Lambda}$,  we consider $G = (-\Delta_{\Lambda})^{-1}$, with $\Delta_{\Lambda}$ the restriction of $\Delta_{\tilde{\Lambda}}$ to $\Lambda$. This constructs a measure such that for each realisation $x_i = 0$ for all $i \in \partial^{\textup{ex}} \Lambda$.

\subsubsection{Example: the discrete fractional Gaussian field}
Let $\Lambda \subset \mathbb{Z}^d$ be a finite subset, and let $\alpha > 0$ be a fractional exponent. Let $\Delta^{\alpha}_\Lambda$ be the discrete fractional Laplacian on $\Lambda$ with suitable boundary conditions, as defined in \cite{chiarini_constructing_2021}. Letting
\begin{equation}
\label{def:dfgf}
G^\alpha = \left( -\Delta_\Lambda^\alpha \right)^{-1}
\end{equation}
be the inverse of this operator, the discrete fractional Gaussian field (DFGF) with exponent $\alpha$ is a centered Gaussian field $(\phi_i)_{i \in \Lambda}$ whose covariance is given by $G^\alpha$. Since $-\Delta_\Lambda^\alpha$ is a positive-definite operator on $\Lambda$, its inverse $G^\alpha$ exists and is also symmetric and positive-definite. These fields correspond to non-local interactions and generally appear in models of random surfaces, turbulence, and percolation see \cite{lodhia_fractional_2016} for a general reference.

\medskip

For infinite $\Lambda$, $G^{-1}$ is a symmetric positive-definite operator on $\ell^2(\Lambda)$, and $G$ is its inverse, interpreted as the covariance operator of the field. The operator $G$ may be defined via limits of finite-volume approximations, and in some cases, a Gaussian field can be constructed as a generalized random field, meaning that it is only defined in the sense of its correlations with test functions. For the infinite domain $\Lambda = \mathbb{Z}^d$, the DGFF is formally defined by taking the infinite-volume limit of the measures $\mu_{\Lambda, G^\textup{green}}$ on finite subsets. The resulting field $h = (h_i)_{i \in \mathbb{Z}^d}$ is a generalized Gaussian field with covariance  
\begin{equation}
    \bE[ h_i h_j ] = G^\textup{green}(i, j),
\end{equation}
where $G^\textup{green}(i, j)$ is the Green's function of $-\Delta$ on $\mathbb{Z}^d$. In dimension $d \geq 3$, the Green's function is summable, and the DGFF is a well-defined Gaussian field. In $d=2$, the DGFF does not define a proper function-valued field but can be interpreted as a distribution-valued field, refer to Section \ref{s:fock-space} for more on this point and for the precise definition of generalized Gaussian field. Another example of Gaussian field defined in $\Z^d$ is the gradient of the DGFF, where $G$ is the (double) derivative of the Green's function, see Equation \eqref{def:gradient-dgff} below for the precise definition. Classical results \cite{dembo_stochastic_2005} ensure that this operator is well defined for any $d\geq 2$.

\medskip

We are ready to state the main theorem.

\begin{theorem}
\label{thm:wick-analytic}
    Let $d\geq 2$ and $\Lambda$ be a finite or infinite subset of $\Z^d$ and $\mu_{\Lambda, G}$ a Gaussian measure as defined in Equation \eqref{def:discrete-gaussian-field} and $\phi$ a realisation of the field. For any $f:\R \to \R$ analytic function $f(x) = \sum_n a_n x^n$ and any $A =\{x_1, \dots, x_{|A|}\} \subset \Lambda$ finite subset, it holds that
    \begin{equation}
        \bE\left[\prod_{x\in A} :f(\phi(x)):\right] = \sum_{n_1, \dots, n_{|A|}} \left(\prod_{i=1}^{|A|} a_{n_i} \right)\sum_{q\in MG_0(A, (n_i)_i)} \delta(q) \prod_{i, j \in A, \, i\leq j} G(i,j)^{q_{ij}},
    \end{equation}
    where $\delta(q)$ is defined in Equation \eqref{def:delta-q}.
    In particular, it holds that
        \begin{equation}
        k(:f(\phi(x)): \textup{ for } x\in A) = \sum_{n_1, \dots, n_{|A|}} \left(\prod_{i=1}^{|A|} a_{n_i} \right)\sum_{\substack{q\in MG_0(A, (n_i)_i)\\ \textup{connected}}} \delta(q) \prod_{i, j \in A, \, i \leq j} G(i,j)^{q_{ij}}.
    \end{equation}
\end{theorem}

Before giving the proof, let's see an application of Theorem \ref{thm:wick-analytic}. Let $\phi$ be a realisation of the DGFF and $f$ be given by $f(\phi(x)) = \sum_{i=1}^d \left(\phi(x+e_i) - \phi(x)\right)^2 = \norm{\nabla \phi(x)}^2$, i.e., the gradient squared of the field in each point. As the gradient of the DGFF is a Gaussian field itself, $\left(\phi(x+e_i) - \phi(x)\right)^2$ is already a polynomial of the field and we can write that
\begin{equation}
    \bE\left[\prod_{x\in A} :f(\phi(x)):\right] = \sum_{q\in MG_0(A, 2)} \delta(q) \sum_{\eta:B\to \mathcal{E}} \prod_{i, j \in A, \, i\leq j}  \nabla^{(1)}_{\eta(i)} \nabla^{(2)}_{\eta(j)} G^\textup{green}(i,j)^{q_{ij}},
\end{equation}
with $\mathcal{E}=\{e_1, \dots, e_d\}$ the standard orthonormal basis of $\R^d$. Given Remark \ref{rem:duality-permutation-mg}, we observe that the above representation is equivalent to the one given in \cite[Theorem 1]{cipriani_properties_2023}, i.e.,
\begin{equation}
    \sum_{\sigma \in \mathcal{S}^0(A)} c(\sigma) \sum_{\eta:B\to \mathcal{E}} \prod_{j \in A} \, \nabla^{(1)}_{\eta(j)} \nabla^{(2)}_{\eta(\sigma(j))} G^\textup{green}(j,\sigma(j)),
\end{equation}
where $c(\sigma)$ takes into account the multiplicity of different permutations yielding the same Feynman diagram\footnote{Contrary to what stated in \cite[Proof of Theorem 1]{cipriani_properties_2023}, every Feynman diagram is not associated with one single permutation but with $2^{|\pi|}$, where $|\pi|$ is the number of cycles: as diagrams are undirected, for every cycle $\sigma$, the cycle $\sigma^{-1}$ yields the same diagram. The correct factor in \cite[Theorem 1]{cipriani_properties_2023} should be $2^{|B|-2}$ instead of $2^{|B|-1}$ and similarly in \cite[Lemma 2]{cipriani_properties_2023}, the factor $(1/2)^{|\pi|}$ should be corrected with $(1/2)^{|\pi|+ 1}$.}.

\begin{proof}
    Using the definition of $f$, i.e., $f(x)=\sum_{n\in \N} a_n x^{n}$, we have that
\begin{equation}
\begin{split}
    &\prod_{x\in A} :f(\phi(x)): \; = \prod_{i=1}^{|A|} : \Big(\sum_{n_i\in \N} a_{n_i} \phi(x_i)^{n_i}\Big): \\
    &=\prod_{i=1}^{|A|} \sum_{n_i\in \N} a_{n_i} :\phi(x_i)^{n_i}: \; = \sum_{n_1, \dots, n_{|A|}} \prod_{i=1}^{|A|} a_{n_i} :\phi(x_i)^{n_i}:.
\end{split}
\end{equation}
By taking the expected value, we obtain that
\begin{equation}
    \bE\left[ \prod_{x\in A} :f(\phi(x)): \right] = \sum_{n_1, \dots, n_{|A|}} \left(\prod_{i=1}^{|A|} a_{n_i} \right) \bE\left[\prod_{i=1}^{|A|}:\phi(x_i)^{n_i}:\right].
\end{equation}
We can apply Theorem \ref{thm:wick-graph} with $Y_i=\phi(x_i)^{n_i}$ for $i=1, \dots, |A|$ to obtain that
\begin{equation}
    \bE\left[ \prod_{x\in A} :f(\phi(x)): \right] = \sum_{n_1, \dots, n_{|A|}} \left(\prod_{i=1}^{|A|} a_{n_i} \right) \sum_{q\in MG_0(A, (n_i)_i)} \delta(q) \prod_{1\leq i\leq j\leq |A|} \bE[\phi(x_i)\phi(x_j)]^{q_{ij}}.
\end{equation}
Using the characterization of the Gaussian field with $G$ covariance matrix, the first part of the proof is concluded.

For the second part, recall that the joint cumulant is a multilinear form, which means that 
\[
k\left(:f(\phi(x_1)):, \dots, :f(\phi(x_{|A|})):\right) = \sum_{n_1, \dots, n_{|A|}} \left(\prod_{i=1}^{|A|} a_{n_i}\right) k\left(:\phi(x_1)^{n_1}:, \dots, :\phi(x_{|A|})^{n_{|A|}}:\right).
\]
We now claim that
\begin{equation}
k\left(:\phi(x_1)^{n_1}:, \dots, :\phi(x_{|A|})^{n_{|A|}}:\right) =     \sum_{\substack{q\in MG_0(A, (n_i)_i)\\ \textup{connected}}} \delta(q) \prod_{i, j \in A, \, i\leq j} G(i,j)^{q_{ij}}
\end{equation}
which concludes the proof.

The claim can be proved using Theorem \ref{thm:cumulants} below by choosing $g=G$, $f=\delta$ and observing that $k$ is the cumulant function of $F(B):=\bE\left[\prod_{i=1}^{|B|}:\phi(x_i)^{n_i}:\right]$ with $B\subset \N$ finite set.
\end{proof}

\subsection{On the properties of the cumulants}

We conclude this section with a useful result about the cumulant function that will be used further  in Section \ref{s:bosonic-fermionic}. The authors could not find anything similar in the literature, the closest result is given by \cite[Proposition 3.2.1]{peccati_wiener_2011} which is linking joint cumulants of a vector to product of cumulants of its entries.

For $n$ positive integer, let $MG([n])$ denote the set of all multigraphs on $n$ elements and $MG_0([n])$ the set of all multigraphs on $n$ elements with no self-loop. For every $n$, let $f_n: MG([n]) \to \R$ be a function defined on multigraphs. Suppose that for every multigraph $q \in MG([n])$ that can be uniquely decomposed into $l$ connected multigraphs $q_j \in MG([n_j])$, with $n_j$ the size of the matrix $q_j$, it holds that $f_n(q) = \prod_{j=1}^l f_{n_j}(q_j)$. By abuse of notation we will drop the subscript in $f$. One example is $f(q)=\delta(q)$ from Theorem \ref{thm:wick-graph} and another one would be $f(q)=(-1)^{\sum_{j=1}^l (n_j-1)}$.

\begin{theorem}
\label{thm:cumulants}
 For any symmetric function $g: \N \times \N \to \R$, define $F_g: \mathcal{P}(\N) \to \R$ on a finite set $B\subset \N$ by
\begin{equation}
    F_g(B) \overset{\text{def}}= \sum_{q \in MG(B)} f(q) \prod_{i,j \in B, i \leq j} g(i, j)^{q_{ij}}.
\end{equation}
Let $k_g: \mathcal{P}(\N) \to \R$ be defined for every finite $A\subset\N$ by
\begin{equation}\label{eq:cum_fct_F}
    k_g(A) \overset{\text{def}}= \sum_{\pi \in \Pi(A)} (|\pi|-1)! (-1)^{|\pi| - 1} \prod_{B\in \pi} F_g(B).
\end{equation}
Then
\begin{equation}
    k_g(A) = \sum_{\substack{q \in MG(A)\\\textup{connected}}} f(q) \prod_{i,j\in A, i\leq j} g(i,j)^{q_{ij}}.
\end{equation}
The same result holds when $MG$ is replaced by $MG_0$ or $\bigcup_{n} MG([n], (l_i)_{i\in[n]})$ with an arbitrary sequence of $(l_i)_{i\in [n]}$.
\end{theorem}

In the sequel, we call $k_g$ the \textit{cumulant function} of $F_g$.

\begin{proof}
We drop the subscript $g$ for notation convenience. The proof is based on induction on the size of the set $A\subset \N$ and the M\"obius inversion formula
\begin{equation}
    k(A) = F(A) - \sum_{\pi \in \Pi(A) : |\pi|>1} \prod_{B\in \pi} k(B),
\end{equation}
which is  equivalent to the moments-cumulants relation. For $|A| =1$ the result is trivially true as there exists only one partition of size 1. Observe that, in the case of $MG_0$, we consider the singleton to be a connected graph and thus the result remains true but with $k(A)=0$ if there is more than one edge.

By the induction assumption, we can rewrite $\sum_{\pi \in \Pi(A) : |\pi|>1} \prod_{B\in \pi} k(B)$ as
\begin{equation}
    \sum_{\pi \in \Pi(A) : |\pi|>1} \prod_{B\in \pi} \sum_{\substack{q \in MG(B)\\\textup{connected}}} f(q) \prod_{i,j\in B, i\leq j} g(i,j)^{q_{ij}}.
\end{equation}
By exchanging $\prod_{B\in \pi}$ with $\sum_{\substack{q \in MG(B)\\\textup{connected}}}$, we can further rewrite this quantity as
\begin{equation}
\begin{split}
    & \sum_{\pi \in \Pi(A) : |\pi|>1} \sum_{\substack{q \in MG(A)\\ q = q_1\cdots q_{|\pi|}}} \prod_{l=1}^{|\pi|}  f(q_l) \prod_{i,j\in A, i\leq j} g(i,j)^{q_{ij}}\\
    & = \sum_{\substack{q \in MG(A)\\ \textup{disconnected}}} f(q) \prod_{i,j\in A, i\leq j} g(i,j)^{q_{ij}},
\end{split}
\end{equation}
where we have used the multiplicative property of $f$. In other words, the contributions from all nontrivial partitions (those with \(|\pi|>1\)) cancel the contributions coming from the disconnected multigraphs. The proof is concluded by subtracting this term from $F(A)$.
\end{proof}

There is a known equivalence between permutations and multigraphs with precisely two edges per node, recall Remark \ref{rem:duality-permutation-mg}. Another characterization of the previous theorem can be derived for the more familiar symmetric case. Given $A\subset \N$, let $\mathcal{S}(A)$ be the symmetric group on $A$ and $\mathcal{S}_{\textup{cycl}}(A) \subset \mathcal{S}(A)$ the subset of all cyclic permutations. Let $\tilde{f}$  be a function defined on permutations such that for every permutation $\sigma$ with unique cycle decomposition given by $\sigma = \sigma_1 \cdots \sigma_l$, it holds that $\tilde{f}(\sigma) = \prod_{j=1}^l \tilde{f}(\sigma_j)$. One example of such $\tilde{f}$ is $\tilde{f}(\sigma)=\text{sign}(\sigma)$.

\begin{cor}
\label{cor:cumulants}
For any symmetric function $g: \N \times \N \to \R$, define $\tilde{F}_g: \mathcal{P}(\N) \to \R$ on a finite set $B\subset \N$ by
\begin{equation}
    \tilde{F}_g(B) \overset{\text{def}}= \sum_{\sigma \in \mathcal{S}(B)} \tilde{f}(\sigma) \prod_{i\in B}g(i, \sigma(i)).
\end{equation}
Let the cumulant function $k_g$ of $\tilde{F}_g$ be defined as in \eqref{eq:cum_fct_F}.
Then
\begin{equation}
    k_g(A) = \sum_{\sigma \in \mathcal{S}_{\textup{cycl}}(A)} \tilde{f}(\sigma) \prod_{i\in A} g(i, \sigma(i)).
\end{equation}
The same result holds when $\mathcal{S}$ is substituted by $\mathcal{S}_0$, the space of permutations without fixed point.
\end{cor}

\subsection{Scaling limit of the $k$-point function}

Discrete Gaussian fields usually have a continuous counterpart, e.g., the DGFF has a natural continuous limit, the continuous Gaussian free field. In this subsection, we investigate this relation further at the level of the $k$-point function.

Let $U \subset \R^d$ be a connected open set, and let $g_U: U \times U \to \R$ be a symmetric positive-definite function. A \textit{continuous Gaussian field} on $U$ is a random distribution $\overline{\phi}$ defined on test functions $f \in C_c^{\infty}(U)$, which has mean zero and satisfies the correlation identity  
\begin{equation}
\label{def:continuous-gaussian-field}
    \mathbb{E}\left[\langle \overline{\phi}, f \rangle \langle \overline{\phi}, g \rangle \right] = \int_U \int_U g_U(x,y) f(x) g(y) \dd x \dd y, \quad \text{for all } f, g \in C_c^\infty(U).
\end{equation}
The function $g_U$ is known as the covariance kernel of the Gaussian field.

A classical example is the \textit{continuum Gaussian free field} (GFF), which arises when $g^{\text{green}}_U$ is the Green's function of the Laplace operator on $U$ with Dirichlet boundary conditions.

\begin{rem}
    In \cite{janson_gaussian_1997, KangMakarov} a Gaussian field indexed by some real Hilbert space $(H, \langle \cdot, \cdot\rangle_H)$, is defined as an isometry
    \begin{equation}
    \label{def:gaussian-field-abstract}
        H \to L^2(\Omega, \mathcal{F}, \bP), \quad h\mapsto\xi_h
    \end{equation}
    such that the image consists of centered Gaussian variables, where isometry means that
\begin{equation}
\label{def:isometry}
\bE[\xi_h \xi_{h'}] = \langle h, h' \rangle_H, \quad \text{ for } h,h'\in H.
\end{equation}
Definition \eqref{def:gaussian-field-abstract} is an abstract formulation equivalent to both the discrete Gaussian field defined in \eqref{def:discrete-gaussian-field} and the continuous Gaussian field defined by \eqref{def:continuous-gaussian-field}. Given a Gaussian measure with covariance operator $G$ (or covariance kernel $g_U$), one can define the Hilbert space $H$ as the closure of the space of test functions with respect to the inner product
    \[
    \langle f, g \rangle_H = \langle f, G^{-1}g \rangle.
    \]
    Then, the mapping
    \[
    h \mapsto \langle \phi, h \rangle,
    \]
    defines an isometry from $H$ into $L^2$, and each $\langle \phi, h \rangle$ is a centered Gaussian variable with variance $\|h\|_H^2$.
 Conversely, if one starts with an isometry as in \eqref{def:gaussian-field-abstract}, then the collection $\{\xi_h\}_{h\in H}$ of centered Gaussian random variables, with covariance structure
    \[
    \mathbf{E}[\xi_h \xi_{h'}] = \langle h, h' \rangle_H,
    \]
    uniquely determines a Gaussian measure on the dual space of $H$. In this sense, the Gaussian field is completely determined by its covariance, encoded in the inner product on $H$.
\end{rem}

\medskip

Suppose we have a sequence of discrete Gaussian fields \((\mu_\varepsilon)_{\varepsilon>0}\) defined on 
\[
U_\varepsilon := \varepsilon^{-1} U \cap \Z^d,
\]
with covariance matrix $G_\varepsilon$. We assume that there exists a sequence of positive numbers \(\eta(\varepsilon)\) such that for any \(x,y\in U\),
\begin{equation}
\label{hyp:covariance-convergence}
\lim_{\varepsilon \to 0} \eta(\varepsilon) \, G_\varepsilon\Bigl(x^{(\varepsilon)}, y^{(\varepsilon)}\Bigr) = g_U(x,y),
\end{equation}
where \(x^{(\varepsilon)} = \lfloor x/\varepsilon \rfloor\) and \(y^{(\varepsilon)} = \lfloor y/\varepsilon \rfloor\) are the corresponding points in $U_\varepsilon$. We are now ready to state the main result of this subsection.

\begin{proposition}
\label{pro:non-triv-cumulants}
  Let $U\subset \R^d$ with $d\geq 2$. Let $x^{(1)},\dots, x^{(k)}$ be disjoint points in $U$ and define $x^{(j)}_\varepsilon := \lfloor x/\varepsilon \rfloor \in U_\varepsilon$ for $j\in [k]$. Let $(\mu_\varepsilon)_{\varepsilon>0}$ be a sequence of Gaussian measures on $U_\varepsilon$ such that Equation \eqref{hyp:covariance-convergence} holds. Then, for every $f$ analytic function $f(x)=\sum_n a_n x^n$, the field $f(\phi)$ satisfies
  
  \begin{align}
&\lim_{\varepsilon \rightarrow 0} \eta(\varepsilon)^k \bE\left[ \prod_{j=1}^k : f \left(\phi(x^{(j)}_\varepsilon) \right): \right]\\
&= \sum_{n_1, \dots, n_{k}} \left(\prod_{i=1}^{k} a_{n_i} \right)\sum_{q\in MG_0([k], (n_i)_i)} \delta(q) \prod_{1\leq i \leq j \leq k} g_U(x^{(i)},x^{(j)})^{q_{ij}}.
\label{eq:continuous-correlation}
\end{align}
A similar statement holds for the cumulants.
\end{proposition}

\begin{proof}
    The proof follows easily by Theorem \ref{thm:wick-analytic} and exchanging the limit for $\varepsilon>0$ with the sum and using hypothesis \eqref{hyp:covariance-convergence}.
\end{proof}

For Gaussian measures, Equation \eqref{hyp:covariance-convergence} guarantees the convergence of all finite-dimensional marginals. If the sequence $(\mu_\varepsilon)_{\varepsilon>0}$ is tight (in a suitable space), then $\mu_\varepsilon$ weakly converges to the continuum Gaussian field $\mu_U$, see, e.g., \cite{dembo_stochastic_2005}. In particular, Equation \eqref{eq:continuous-correlation} corresponds to the correlations of a continuous Gaussian field: we will investigate this behavior in more detail in Section \ref{s:fock-space}. We now review a few fields satisfying Equation \eqref{hyp:covariance-convergence} for which Proposition \ref{pro:non-triv-cumulants} applies.

\subsubsection{Example: discrete Gaussian Free Field and its gradient}
For e.g. $d\geq 2$, it is well known (see e.g. \cite{biskup20}) that, for an appropriate rescaling \(\eta(\varepsilon) = -\log(\varepsilon)\) for $d=2$ and \(\eta(\varepsilon) = \varepsilon^{d-2}\) for $d\geq 3$, the discrete Green's function satisfies:  
\[
\lim_{\varepsilon \to 0} \eta(\varepsilon) G^\textup{green}_\varepsilon(x^{(\varepsilon)}, y^{(\varepsilon)}) = g^\textup{green}_U(x,y)
\]
where \( g^\textup{green}_U(x,y) \) is the Green’s function of the continuum Laplacian \( -\Delta \) on \( U \) with Dirichlet boundary conditions.

The gradient of the DGFF, defined by taking the discrete gradient of the corresponding Green's function, behaves in a similar way. One can show \cite[Lemma 5]{cipriani_properties_2023} that for \(\eta(\varepsilon) = \varepsilon^{d}\), it holds that
\begin{equation}
\label{def:gradient-dgff}
\lim_{\varepsilon \to 0} \eta(\varepsilon) \nabla^{(1)}_i \nabla^{(2)}_j G^\textup{green}_\varepsilon(x^{(\varepsilon)}, y^{(\varepsilon)}) = \partial^{(1)}_i \partial^{(2)}_j g^\textup{green}_U(x,y),
\end{equation}
where $\nabla^{(1)}_i$ represents the discrete gradient acting on the first variable in the direction $i$ and $\partial^{(1)}_i$ its counterpart in the continuous, and similarly for $\nabla^{(2)}$ and $\partial^{(2)}$ for the second variable.

\subsubsection{Example: discrete membrane model}
For $d\geq 4$, the \emph{discrete membrane model} is defined similarly to the DGFF but is based on the inverse of the discrete bi-Laplacian \( (\Delta_\varepsilon)^2 \) rather than the discrete Laplacian, see \cite{cipriani_scaling_2018}. The covariance matrix $G^{\mathrm{bi}}_\varepsilon(x^{(\varepsilon)}, y^{(\varepsilon)})$ is given by the Green’s function of \( (\Delta_\varepsilon)^2 \) with Dirichlet boundary conditions. Under the rescaling \( \eta(\varepsilon) = -\log(\varepsilon) \) for $d=4$ and \( \eta(\varepsilon) = \varepsilon^{d-4} \) for $d\geq 5$, this function converges to the Green's function of the \emph{continuum bi-harmonic operator} \( \Delta^2 \):  
\[
\lim_{\varepsilon \to 0} \eta(\varepsilon) G^\mathrm{bi}_\varepsilon(x^{(\varepsilon)}, y^{(\varepsilon)}) = \int_U g^{\textup{green}}(x,z)g_U^{\textup{green}}(z,y) \dd  z.
\]

\subsubsection{Example: discrete fractional Gaussian fields}
For \( \alpha >0 \) such that $\alpha \neq 2$, it was proven e.g. in \cite{chiarini_constructing_2021} that, the covariance matrix $G^\alpha_\varepsilon$ of a DFGF, recall Equation \eqref{def:dfgf}, satisfies the hypothesis and converges under appropriate rescaling \( \eta(\varepsilon) = -\log(\varepsilon) \) if $\alpha=d/2$, \( \eta(\varepsilon) = \varepsilon^{\alpha}-d/2 \) if $\alpha > d/2$, resp. \( \eta(\varepsilon) = \sqrt{\log(\varepsilon)} \) if $\alpha < d/2$ to the \emph{fractional Gaussian field} associated with the continuum operator \( (-\Delta)^{-\alpha} \). For instance, when \( \alpha = 1 \), one recovers the Gaussian free field; when \( \alpha = 2 \), we obtain the bi-harmonic field and, for fractional values, we get non-local Gaussian fields with long-range interactions.

\section{Connection to Fock space fields}
\label{s:fock-space}

In this section, we connect the analytic functions of Gaussian fields to Fock space fields. Our presentation closely follows \cite[Section IV]{janson_gaussian_1997} and \cite[Section 1]{KangMakarov}, notably the construction of Fock spaces is done via symmetric tensor powers of Gaussian Hilbert spaces.

\subsection{Fock space fields}

We recall some standard facts on tensor products of Hilbert spaces (with real or complex scalars).
If $H_1$ and $H_2$ are two Hilbert spaces, their \textit{tensor product} $H_1 \otimes H_2$ is defined to be a Hilbert space together with a bilinear map $H_1 \times H_2 \to H_1 \otimes H_2$, denoted by $(f_1, f_2) \mapsto f_1 \otimes f_2 \in H_1 \otimes H_2$, such that
\[
\langle f_1 \otimes f_2, g_1 \otimes g_2 \rangle_{H_1 \otimes H_2} \overset{\text{def}}= \langle f_1, g_1 \rangle_{H_1} \langle f_2, g_2 \rangle_{H_2}.
\]
The tensor product of several Hilbert spaces is defined in the same way.

For $n$ positive integer, the \textit{symmetric tensor power} $H^{\odot n}$ of a Hilbert space is similarly defined to be a Hilbert space together with a symmetric multilinear map $H \times \cdots \times H \to H^{\odot n}$, denoted by $(f_1, \dots, f_n) \mapsto f_1 \odot \cdots \odot f_n$, such that
\[
\langle f_1 \odot \cdots \odot f_n, g_1 \odot \cdots \odot g_n \rangle \overset{\text{def}}= \sum_{\pi \in \mathcal{S}([n])} \prod_{i=1}^{n} \langle f_i, g_{\pi(i)} \rangle.
\]
We let $H^{\odot 0}$ be the one-dimensional space of scalars. Note also that $H^{\odot 1} = H$. It is well-known that the algebraic direct sum $\sum_{n=0}^{\infty} H^{\odot n}$ is a graded commutative algebra, usually called the \textit{symmetric tensor algebra} of $H$. Its completion, the Hilbert space 
\begin{equation}
    \Gamma(H) \overset{\text{def}}= \bigoplus_{n=0}^{\infty} H^{\odot n},
\end{equation}
is called the \textit{(symmetric) Fock space} over $H$.

We now compare these definitions with the Wick product defined in Equation \eqref{def:fock-wick-product}. Consider the map $H^{\odot n} \rightarrow H^{:n:}$, defined by
\[
X_1 \odot \ldots  \odot X_n \mapsto :X_1 \cdot \ldots \cdot X_n:,
\]
it is well known \cite[Theorem 4.1]{janson_gaussian_1997} that it defines a Hilbert space isometry of $H^{\odot n}$ onto $H^{:n:}$, that can be extended to the symmetric tensor algebra, i.e.,  to establish the \textit{Wiener chaos decomposition}
\[
L^2(\Omega, \mathcal{F} , \bP) \cong \Gamma(H).
\]

\medskip

Let $U\subset\R^d$ be a connected open set and  $\mathcal{H}:=\mathcal{H}^1_0(U)$ be the Sobolev space with Dirichlet inner product, i.e., the completion of $C^{\infty}_c(U)$ function w.r.t the classical Sobolev norm. Let $\overline{\phi}$ be a continuous Gaussian Field on $\mathcal{H}$, defined as in Equation \eqref{def:continuous-gaussian-field} with $g_U$ the covariance operator. We further define the derivative $\partial^{\alpha} \overline{\phi}$ in the direction $\alpha\in \{1, \dots, d\}$ as a Gaussian distributional field in the following sense
\[
\langle \partial^{\alpha} \overline{\phi}, f\rangle = \langle \overline{\phi}, \partial^{\alpha} f\rangle,
\]
for $f\in C^{\infty}_c(U)$.

To characterize the elements of $\Gamma(H)$ we need the concept of Fock space fields and correlation functionals. We call \emph{basic Fock space fields} the formal expressions that can be written as Wick's products of $\overline{\phi}$ and its derivatives; examples are
\[
1, \quad \overline{\phi} \odot \overline{\phi}, \quad \partial \overline{\phi} \odot \overline{\phi}, \quad \partial^2 \overline{\phi} \odot \overline{\phi} \odot \partial \overline{\phi}, \quad \dots 
\]
A \emph{Fock space field} $X\in \Gamma(H)$ is defined as a (countable) linear combination over $\bbC$ of basic Fock space fields $(X_n)_n$, i.e.,
\[
X=\sum_{n} f_n X_n
\]
where $f_n \in C^{\infty}_c(U)$ are called \textit{basic field coefficients}. The basic Fock space fields $\overline{\phi}$ and $\partial \overline{\phi}$ are clearly important examples of Fock fields. The exponential $\text{Exp} ( \odot \alpha \, \overline{\phi})$, can be formally defined for $\alpha \in \mathbb{R}$ by
\[
\text{Exp} ( \odot \alpha \, \overline{\phi}) \overset{\text{def}}= \sum_{n=0}^{\infty} \frac{\alpha !}{n!} (\overline{\phi})^{\odot n}.
\]
More generally, if $f$ is an analytic function $f(x)=\sum_n a_n x^n$, we denote the corresponding Fock space field as
\begin{equation}
    f\left(\odot \overline{\phi}\right) \overset{\text{def}} = \sum_{n=0}^{\infty} a_n (\overline{\phi})^{\odot n}.
\end{equation}

\subsection{Correlation functionals and scaling limits of the $k$-point correlation}
The correlation functionals of Fock space fields will be connected with the scaling limit of the cumulants of a discrete Gaussian field. Let $z_1,\dots, z_n \in U$, a \emph{basic correlation functional} is defined by
\begin{equation}
\label{def:basic-corr-func}
\mathcal{X} = X_1(z_1) \odot \cdots \odot X_n(z_n),
\end{equation}
where the $z_i$'s are not necessarily distinct and $X_i$’s are basic Fock space fields. We also include the constant 1 to the list of basic functionals.
The set $S(\mathcal{X})=\{z_1,\ldots, z_n\}$ is known as the set of \emph{nodes} of the functional $\mathcal{X}$.

Let $k\geq 2$ be a positive integer and consider the following collection of correlation functionals for $1\leq j \leq k$
\[
\mathcal{X}_j = X_{j_1}(z_{j_1}) \odot \cdots \odot X_{j_{l_j}}(z_{j_{l_j}})
\]
with pairwise disjoint pairs of nodes $S(\mathcal{X}_j)$. The \textit{tensor of the correlation functionals} $\mathcal{X}_1, \dots \mathcal{X}_k$ is defined as
\begin{equation}
    \label{def:tensor}
\mathcal{X}_1 \bullet \cdots \bullet \mathcal{X}_k \overset{\text{def}}= \sum_{\gamma} \prod_{\{u,v\} \in E_{\gamma}} \bE[X_u(x_u)X_v(x_v)],
\end{equation}
where the sum runs over $\gamma$ complete Feynman diagram  such that there is no edge connecting two points of the same node set. $E_{\gamma}$ denotes the edges of $\gamma$, as in Section \ref{s:wick}.

We are ready to state the main result of this section.
\begin{proposition}
\label{pro:non-triv-cumulants-fock}
Let $\phi$ be a Gaussian field on $U_\varepsilon$ for which Equation \eqref{hyp:covariance-convergence} is satisfied and let $\overline{\phi}$ be a continuous Gaussian field with covariance operator given by $g_U$. Let $k\geq 2$ be a positive integer, for $j=1, \dots, k$ and $x^{(1)},\dots,x^{(k)}$ disjoint points in $U$. Let
\begin{equation}
\mathcal{Y}_j = f(\odot \overline{\phi}) (x^{(j)}),
\end{equation}
where $f$ is a given analytic function. It holds that
\begin{equation}
\lim_{\varepsilon \rightarrow 0} \eta(\varepsilon)^k \bE \left[ \prod_{j=1}^k :f \left(\phi(x^{(j)}_\varepsilon) \right): \right] = \mathcal{Y}_1 \bullet  \cdots \bullet \mathcal{Y}_k.
\end{equation}
\end{proposition}

\begin{proof}
From Theorem \ref{thm:wick-analytic} and Proposition \ref{pro:non-triv-cumulants}, we know that
\begin{equation}
\begin{split}
    &\lim_{\varepsilon \rightarrow 0} \eta(\varepsilon)^k \bE \left[ \prod_{j=1}^k :f \left(\phi(x^{(j)}_\varepsilon) \right): \right] \\
    &= \sum_{n_1, \dots, n_{k}} \left(\prod_{i=1}^{k} a_{n_i} \right)\sum_{q\in MG_0([k], (n_i)_i)} \delta(q) \prod_{1\leq i \leq  j \leq k} g_U(x^{(i)},x^{(j)})^{q_{ij}}.
\end{split}    
\end{equation}
On the right hand side, similar to Theorem \ref{thm:wick-graph}, we recognize the Wick product  of $(\mathcal{Y}_j)_j$ as defined in Equation \eqref{def:tensor}. In particular
\begin{equation}
\begin{split}
    &\mathcal{Y}_1 \bullet  \cdots \bullet \mathcal{Y}_k \\
    &=\sum_{n_1=0}^{\infty} a_{n_1} (\overline{\phi})^{\odot n_1}(x^{(1)}) \bullet \cdots \bullet \sum_{n_k=0}^{\infty} a_{n_k} (\overline{\phi})^{\odot n_k}(x^{(k)})\\
    &=\sum_{n_1, \dots, n_k} \left( \prod_{l=1}^k a_{n_l}\right) (\overline{\phi})^{\odot n_1}(x^{(1)}) \bullet \cdots \bullet (\overline{\phi})^{\odot n_k}(x^{(k)})\\
    &=\sum_{n_1, \dots, n_k} \left( \prod_{l=1}^k a_{n_l}\right) \sum_{\gamma} \prod_{\{u,v\}\in E_\gamma} \bE[X_u(x_u)X_v(x_v)],
\end{split}
\end{equation}
where the sum run overs $\gamma$ complete Feynman diagram  such that there is no edge connecting two points of the same node set. By using the characterization of $\overline{\phi}$, we obtain that
\[
\mathcal{Y}_1 \bullet  \cdots \bullet \mathcal{Y}_k =\sum_{n_1, \dots, n_{k}} \left(\prod_{i=1}^{k} a_{n_i} \right)\sum_{q\in MG_0([k], (n_i)_i)} \delta(q) \prod_{1\leq i \leq  j \leq k} g_U(x^{(i)},x^{(j)})^{q_{ij}}
\]
and the proof is concluded.
\end{proof}

\subsubsection{Example: exponential of the gradient of the GFF}
We apply the previous ideas to the gradient of the discrete GFF, generalizing the result already obtained in \cite{cipriani_properties_2023}. Let $\phi$ be the discrete GFF on $U_\varepsilon$ and denote the Wick product of the discrete gradient by
\begin{equation}
    \Phi_\varepsilon(e, x) := : (\phi(x+e) - \phi(x)) :,
\end{equation}
where $x\in U_\varepsilon$ and $e\in \{e_1, \dots, e_d\}$ is a given direction. Fix $\alpha\in \R$, for every direction $e_i$ and every $x\in U$, and let
\begin{equation}
    \mathcal{Z}_{e_i, x} = \exp \left ( \odot \alpha \partial^{e_i} \overline{\phi}(x) \right ) = \sum_{n=0}^\infty \frac {\alpha!}{n!} \left( \partial^{e_i} \overline{\phi} (x)\right)^{\odot n}.
\end{equation}
Then, for every $(x^{(1)}_\varepsilon, e^1), \dots, (x^{(k)}_\varepsilon, e^k)$ couples of distinct points $(x^{(1)}_\varepsilon, \dots, x^{(k)}_\varepsilon)$ and corresponding directions $(e^1, \dots, e^k)$ (not necessarily distinct), let $x_i$ be the limit of $x^{(i)}_\varepsilon$ as $\varepsilon$ goes to zero. It holds that
\begin{equation}
    \lim_{\varepsilon \rightarrow 0} \varepsilon^{-d\cdot k} \sum_{n_1, \dots, n_k} \frac {\alpha^{\sum_i n_i}} {n_1! \dots n_k!} \bE \left[\prod_{j=1}^k  \Phi^{n_j}_{\varepsilon}(x^{(j)}_\varepsilon, e^j) \right] = \mathcal{Z}_{e^1, x_1} \bullet \cdots  \bullet\mathcal{Z}_{e^k, x_k}.
\end{equation}
Similarly one can obtain the scaling limit of the correlation functions for the exponential of the GFF.

\section{Fermionic (discrete) Gaussian fields}
\label{s:bosonic-fermionic}

In this section, we investigate the relation between fermionic and bosonic Gaussian fields by looking at their cumulants. Before giving the main result, let us introduce some basic notation and state the needed algebraic identities regarding discrete fermions and Grassmann-Berezin calculus. The main references for this section are \cite{caraciolo13,chiarini_fermion23}. 

Let $M\in \mathbb{N}$ and let $\xi_1, \ldots,\xi_M$ be a collection of letters and $\Omega_M \overset{\text{def}}=\mathbb{R}[\{\xi_1,\ldots,\xi_M \}]$ the quotient of the Grassmannian algebra generated by the relations for all $i,j\in \{1,\dots,M \}$
\begin{equation}\label{def:rules_para}
\begin{split}
\xi_i \xi_j &= - \xi_j \xi_i, \qquad \xi^2_i  = 0.
\end{split}
\end{equation}
The space $\Omega_M$ is a \textit{polynomial ring over $\mathbb{R}$} with basis equal to 
$\{1, \xi_1,\ldots, \xi_M\}$. Objects from $\Omega_M$ are polynomials $F$ of the form
\begin{equation}
F = \sum_{I \subset [M]} a_I \xi_I , \quad a_I\in \mathbb{R}
\end{equation}
where $I=\{i_1,\ldots i_l \} \subset [M]$ and the symbol $\xi_I$ denotes the product
\[
\xi_I = \xi_{i_1} \ldots \xi_{i_l}.
\]

As done in  \cite[Equation (A.48)]{caraciolo13}, we define the Grassmann-Berezin derivation as the linear map $\partial_{\xi_k}: \Omega_M \rightarrow \Omega_M$ that acts on $\xi_I$ for $I=\{i_1,\ldots,i_l \}$ such that
\begin{equation}
\partial_{\xi_k} \xi_I =  \begin{cases}
 (-1)^{\alpha-1}  \xi_{i_1} \xi_{i_{\alpha-1}} \cdot 1 \cdot \xi_{i_{\alpha+1}} \ldots \xi_{i_l} & \text{ if } i_{\alpha} = k, k \in I,  \\
0 & \text{ if } k\notin I.
\end{cases}
\end{equation}
The integration (see \cite[Equation (A.52)]{caraciolo13}) is defined in a similar fashion, i.e., the linear map from $\Omega_M$ to $\mathbb{R}$, given for $F\in \Omega_M$ by
\begin{equation}
\int \, \partial_{\xi_1} \ldots \partial_{\xi_M}  \, F \overset{\text{def}}=  \prod_{i=1}^M \partial_{\xi_i}  F.
\end{equation}

In the special case in which $M= 2m$, the generators $\xi_1, \dots, \xi_M$ are divided into two sets $\psi_1, \dots, \psi_m$ and $\overline{\psi}_1, \dots, \overline{\psi}_m$ and are usually called "complex" fermions. We have the following result.
\begin{proposition}{\cite[Proposition A.14]{caraciolo13}}
\label{prop:complex-fermionic-gaussian-integration}
Let $C$ be an $m\times m$ matrix with coefficients in $\mathbb{R}$. Then \begin{equation}
\int \partial \psi_1 \partial\overline{\psi}_1 \dots \partial \psi_m \partial\overline{\psi}_m \, \exp \left(\overline{\psi}^T C  \psi \right ) = \det(C),
\end{equation}
and $0$ otherwise.
\end{proposition}
Please note that here $C$ is an arbitrary matrix, e.g., no condition of symmetry is imposed on it. If $C$ is invertible, one can define the (normalised) \emph{fermionic Gaussian field state} $\langle \cdot \rangle_C$ by
\begin{equation}
    \label{def:complex-fermionic-state}
    \langle F \rangle_C = \det(C)^{-1}  \prod_{i=1}^m \partial \psi_i \partial\overline{\psi}_i \, F(\psi, \overline{\psi}) \exp \left(\overline{\psi}^T C  \psi \right ), \quad F\in \Omega_M.
\end{equation}
The following theorem is taken from \cite[Theorem A.16(c)]{caraciolo13}.
\begin{theorem}\label{thm:complex-fermionic-expectation}
Let $C$ be an invertible matrix. For any sequence of (ordered) indices $A \subset [n]$, it holds that
\begin{equation}
\left \langle \prod_{i\in A} \bar{\psi}_i \psi_i \right\rangle_C  = \det((C^{-t})_{AA}),
\end{equation}
where $(C^{-t})_{AA}$ is the submatrix with entries indexed by $A \times A$.
\end{theorem}

In the following subsections we  explore the relation in terms of correlation functions and cumulants of $2r^{th}$-powers ($r\in \mathbb{N}$) of bosonic and  fermionic Gaussian fields.

\subsection{Squares of Gaussian fields}

We apply Theorem \ref{thm:cumulants} to several examples of $F$ issued by different bosonic-fermionic frameworks in the case $p=2$ and show that, although the moments are substantially different, the corresponding cumulants are proportional. In the rest of the section $n$ is fixed, $C=(C(i,j))_{i,j\in[n]}$ and $G=(G(i,j))_{i,j\in[n]}$ denote two symmetric positive-definite $n\times n$ matrices.

\subsubsection{The bosonic case (real)}
Let $(X_1, \dots, X_n)$ be a Gaussian vector with covariance given by the matrix $G$. For $A\subset [n]$, define
\begin{equation}
    F_{\textup{bos}(\R)}(A) := \mathbb{E}\left[ \prod_{i\in A}:  X_i^2 : \right] = \sum_{\substack{\sigma \in \mathcal{S}(A)\\ \sigma=\sigma_1 \cdots \sigma_l}} \prod_{k=1}^l 2^{|\sigma_k|-2} \prod_{i=1}^n G_{i\sigma(i)}
\end{equation}
where the second equality comes from Theorem \ref{thm:wick-analytic} and the following remark. Observe that $F_{\textup{bos}(\R)}$ satisfies the hypothesis of Theorem \ref{thm:cumulants}, with $f(\sigma) = \prod_{l} 2^{|\sigma_l|-2}$ with $\sigma=\sigma_1 \dots \sigma_l$ cycle decomposition and $g(i,j)=G(i,j)$ for all $i,j=1, \dots, n$ (and 0 otherwise).

In a similar fashion, for $A\subset [n]$
\begin{equation}
    F_{\textup{vec}(\R)}(A) := \mathbb{E}\left[ \left(: \prod_{i\in A} X_i: \right)^2 \right] = \sum_{\sigma \in \mathcal{S}(A)} \prod_{i=1}^n G_{i\sigma(i)},
\end{equation}
where the second equality is well-known property of Gaussian random variables (see, e.g., \cite[Theorem 3.9]{janson_gaussian_1997}). It is easy to see that $F_{\textup{vec}(\R)}$ satisfies the hypothesis of Corollary \ref{cor:cumulants}.

\qed

\subsubsection{The bosonic case (complex)}

Let $(Z_1, \dots, Z_n)$ be a vector of complex Gaussian variables with covariance given by $G$, i.e., $\mathbb{E}[Z_i \bar{Z}_j] = G(i,j)$. For $A\subset [n]$, define
\begin{equation}
    F_{\textup{bos}(\bbC)}(A) := \mathbb{E} \left[ \prod_{i \in A} Z_i \bar{Z}_i \right] =\mathbb{E} \left[ \prod_{i\in A} |Z_i|^2 \right].
\end{equation}
A classical application of the Wick theorem for complex Gaussian variables (see, e.g., \cite[Theorem 1]{fassino_computing_2019} expresses the previous expectation in terms of the permanent of $G$, i.e.,
\begin{equation}
    F_{\textup{bos}(\bbC)}(A) = \textup{perm}(G_{AA}) = \sum_{\sigma\in \mathcal{S}(A)} \prod_{i\in A} G_{i\sigma(i)} =  F_{\textup{vec}(\R)}(A),
\end{equation}
where $G_{AA}$ is the submatrix of $G$ where we only take the rows and columns with indices in $A$.

\qed

\subsubsection{The "complex" fermionic case}

Let $\psi_1 \bar{\psi_1}, \dots, \psi_n \bar{\psi}_n$ be $n$ couples of fermionic variables according to the state defined in \eqref{def:complex-fermionic-state} with the matrix $C^{-1}$ . For $A\subset [n]$, define
\begin{equation}
    F_{\textup{ferm}(\bbC)}(A) := \langle \prod_{i\in A} \psi_i \bar{\psi_i} \rangle_{C^{-1}}.
\end{equation}
Theorem \ref{thm:complex-fermionic-expectation} ensures that this quantity is equal to
\begin{equation}
    F_{\textup{ferm}(\bbC)}(A) = \det(C_{AA}) = \sum_{\sigma\in \mathcal{S}(A)} \textup{sign}(\sigma) \prod_{i\in A} C_{i\sigma(i)}.
\end{equation}
The signature of a permutation satisfies the hypothesis on $f$ in Theorem \ref{thm:cumulants}.
\qed

Let $k_{\star}$ be the cumulant function, in the sense of Theorem \ref{thm:cumulants}, of $F_{\star}$ for $\star \in \{ \textup{bos}(\bbR),$ $\textup{vec}(\bbR),\, \textup{bos}(\bbC),\, \textup{ferm}(\bbC)\}$. We are now ready to state the main theorem of this subsection.

\begin{theorem}
\label{thm:cumulants-p=2}
Fix $n\in\N$. There exists a fermionic Gaussian field state such that for any $A\subset [n]$ holds that
\begin{equation}
k_{\textup{bos}(\bbC)}(A) = (-1)^{|A|-1} k_{\textup{ferm}(\bbC)}(A)
\end{equation}
if and only if $C=G$. In particular, if $C=G$, we have that
\begin{equation}
    k_{\textup{vec}(\bbR)}(A) = 2^{|A|-2} k_{\textup{bos}(\bbR)}(A) = k_{\textup{bos}(\bbC)}(A) = (-1)^{|A|-1} k_{\textup{ferm}(\bbC)}(A).
\end{equation}
\end{theorem}

\begin{proof}
Consider $B\subset [n]$, then by the reversing formula for the moments and the assumption, we have that
\begin{equation}
\begin{split}
\left\langle\prod_{i \in B} \psi_i 
\bar{\psi}_i\right\rangle_{C^{-1}} &= \sum_{\pi \in \Pi(B)} \prod_{B\in \pi} k_{\textup{ferm}(
\bbC)}(B) \\
& = \sum_{\pi \in \Pi(B)} \prod_{B\in \pi} (-1)^{|B|-1}k_{\textup{bos}(\bbC)}(B).
\end{split}
\end{equation}
We replace the expression of $k_{\textup{bos}(\bbC)}$ by applying Corollary \ref{cor:cumulants} and obtain that
\begin{equation}
\begin{split}
\det{C_{BB}} = \left\langle\prod_{i \in B}  \psi_i 
\bar{\psi}_i \right\rangle_{C^{-1}} &= \sum_{\pi \in \Pi(B)} \prod_{B\in \pi} \sum_{\sigma \in \mathcal{S}_{\textup{cycl}}([n])} \textup{sign}(\sigma)\, \prod_{i \in B} G(i,\sigma(i)) = \det{G_{BB}},
\end{split}
\end{equation}
where we have used that the signature of a cyclic permutation of size $k$ is equal to $(-1)^{k-1}$. Apply Corollary \ref{cor:cumulants} to every function $F_\star$ we obtain an analogous relation.
\end{proof}

\subsection{$2r$-powers of Gaussian fields for $r>1$}

Consider $(Z_i)_{i\in [n]} \sim N_{\mathbb{C}}(0,G)$  a complex Gaussian vector.  For $r>1, A\subset [n]$ define
\begin{equation}
k^{r}_{\textup{bos}}(A) \overset{\text{def}} = k(|Z_i|^{2r}:\, i\in A)
\end{equation}
the joint cumulants of the $2r$-th powers of $(Z_i)_{i\in[n]}$. It is well known that the Wick theorem for complex Gaussian variables, e.g., \cite[Theorem 1]{fassino_computing_2019}), implies that
\begin{equation}
\label{eq:complex-wick-theorem}
\mathbf{E}\left[\prod_{i\in A} |Z_i|^{2r}\right] = \sum_{\textup{pairings}(A_r)} \prod_{(i,b) \textup{ paired with }(j,c)}G(i,j) = \textup{perm} \left( G^{(r)}(A)\right) 
\end{equation}
where $G^{(r)}(A)$ is a $r|A|\times r|A|$ matrix obtained in the following way: label $(i,1), \dots, (i,r)$ $r$ copies of the index $i\in A$ and let $A_r = \{ (i, a): \; i\in A,\;a\in [r]\}$, then
\begin{equation}
    G^{(r)}(A)_{(i,a), (j,b)} = G(i,j), \quad \text{for } (i,a),(i,b) \in A_r.
\end{equation}
As done in Section \ref{s:wick}, we can express the previous sum in terms of multigraphs.
\begin{lemma}
\label{lem:complex-gaussian-cumulants}
    It holds that
    \begin{equation}
        \mathbf{E}\left[\prod_{i\in A} |Z_i|^{2r}\right] = \sum_{q \in MG(A,2r)} \bar{\delta}(q) \prod_{i,j\in A, i\leq j} G(i,j)^{q_{ij}},
    \end{equation}
    where now (compare to Definition \eqref{def:delta-q})
    \begin{equation}
    \label{def:bar-delta-q}
        \bar{\delta}(q) =\prod_{i\in A} \frac{r!}{2^{q_{ii}}q_{ii}! \prod_{i \neq j} q_{ij}!} 
    \end{equation}
    In particular,
    \begin{equation}
        k^{r}_{\textup{bos}}(A) = \sum_{\substack{q \in MG(A,2r)\\\textup{connected}}} \bar{\delta}(q) \prod_{i,j\in A, i\leq j} G(i,j)^{q_{ij}}.
    \end{equation}
\end{lemma}
\begin{proof}
    The combinatorial factor $\bar{\delta}(q)$ counts the ways to pair each of the $r$ half-edges of vertex $i$ into $q_{ii}$ loops (each loop uses up two half-edges, hence the factor $2^{q_{ii}}q_{ii}!$) and into the $q_{ij}$ with $i\neq j$ half‐edges going off to distinct neighbors.

    The joint cumulants formula is a consequence of Theorem \ref{thm:cumulants}
\end{proof}

\begin{rem}
    Lemma \ref{lem:complex-gaussian-cumulants} is the equivalent of Theorem \ref{thm:wick-graph} for complex jointly Gaussian variables. As the Wick product is not considered here, the class of multigraphs is given by $MG$ instead of $MG_0$. In the case $p=2$, the formula simplifies to the permanent of $G$, see, e.g., \cite[Theorem 1]{fassino_computing_2019}.
\end{rem}

Turning to the fermionic setting, let $C$ be a $(nr)\times (nr)$-dimensional symmetric and invertible matrix. Denote by $\langle \cdot \rangle_C$ the corresponding fermionic Gaussian free field state, recall Equation \eqref{def:complex-fermionic-state}. Consider polynomials of the generators of the form:
\begin{equation}
\varphi^{2r}_i \overset{\text{def}} = \prod_{l=1}^r \psi^{(l)}_i \bar{\psi}^{(l)}_i, \quad i\in [n].
\end{equation}
Analogously as before we define for $A\subset [n]$ the cumulants of the polynomials by
\begin{equation}
k^r_{\textup{ferm}} (A) = k(\varphi^{2r}_i: \; i\in A) \overset{\text{def}} = \sum_{\pi \in \Pi(A)} (|\pi|-1)!(-1)^{|\pi|-1} \prod_{B\in \pi} \left\langle \prod_{i \in B} \varphi^{2r}_i \right\rangle_{C}.
\end{equation}

In the following lemma, for a multigraph $q$ we use the notation $\textup{sign}(q)=(-1)^{\sum_{j=1}^l (n_j-1)}$, where  $n_j$ is the dimension of the submatrix $q_j$ resulting from the unique decomposition in the connected components $q=q_1\dots q_l$.

\begin{lemma}
\label{lem:bos-fer-p>2}
    For any $A\subset [n]$, the following conditions are equivalent
    \begin{enumerate}
        \item The fermionic joint cumulants satisfy
        \begin{equation}
        \label{eq:r-fermionic-cumulants-relation}
            k^r_{\textup{ferm}}(A) = (-1)^{|A|-1} k^r_{\textup{bos}}(A).
        \end{equation}
        \item Let $A_r=\cup_{i\in A} \{ (i-1)r, (i-1)r+1,\dots,ir\}$
        \begin{equation}
        \label{eq:r-fermionic-matrix-minor-condition}
            (\det(C_{A_r, A_r}))^{-1} = \sum_{q\in MG(A,2r)} \textup{sign}(q) \bar{\delta}(q) \prod_{i\leq j, i,j \in A} G(i,j)^{q_{ij}},
        \end{equation}
        where $\bar{\delta}(q)$ is as in Equation \eqref{def:bar-delta-q}.
    \end{enumerate}
    In particular, there exists a fermionic  Gaussian state such that Equation \eqref{eq:r-fermionic-cumulants-relation} holds for every $A\subset [n]$ if and only if the all the $(A_r)_{A\subset[n]}$ principal minors of the matrix $C^{-1}$ satisfy \eqref{eq:r-fermionic-matrix-minor-condition}.
\end{lemma}

\begin{proof}
Consider $B\subset [n]$, then by the reversing formula for the moments and the assumption, we have that
\begin{equation}
\begin{split}
\left\langle\prod_{i \in B} \varphi^{2r}_i \right\rangle_{C} &= \sum_{\pi \in \Pi(B)} \prod_{B\in \pi} k^r_{\textup{ferm}}(B) \\
& = \sum_{\pi \in \Pi(B)} \prod_{B\in \pi} (-1)^{|B|-1}k^r_{\textup{bos}}(B).
\end{split}
\end{equation}
By using Lemma \ref{lem:complex-gaussian-cumulants}, we replace the expression of $k^r_{\textup{bos}}$ and use the definition of $\textup{sign}(q)$ to get
\begin{equation}
\begin{split}
\left\langle\prod_{i \in B} \varphi^{2r}_i \right\rangle_{C} &= \sum_{\pi \in \Pi(B)} \prod_{B\in \pi} \sum_{\substack{q \in MG(B,2r)\\\textup{connected}}} \textup{sign}(q)\, \bar{\delta}(q) \prod_{i,j\in B, i\leq j} G(i,j)^{q_{ij}}.
\end{split}
\end{equation}
Using a similar argument as in the proof of Theorem \ref{thm:cumulants}, the proof is concluded.
\end{proof}

\begin{rem}
    The problem of finding a matrix given its principal minors is a long-standing problem in algebra, see \cite{holtz_hyperdeterminantal_2007, huang_symmetrization_2017, matoui_principal_2021} and references therein. In Lemma \ref{lem:bos-fer-p>2}, not all principal minors are required, but only those related to the indices $A_r=\cup_{i\in A} \{ (i-1)r, (i-1)r+1,\dots,ir\}$ with $A\subset [n]$. For instance, consider the case where $G=I_n$, the identity matrix on $n$ elements, the condition on the principal minors become
    \begin{equation}
        (\det(C_{A_r, A_r}))^{-1} = \sum_{\substack{q\in MG(A,2r)\\ q_{ij}=0, \; i\neq j}} \textup{sign}(q) \bar{\delta}(q) = (-1)^{|A|-1} \left(\frac{r!}{2^{2r} (2r)!}\right)^{|A|},
    \end{equation}
    which, to the authors' knowledge cannot be directly solved.
\end{rem}

\section*{Funding}
F.C. was supported by the NWO (Dutch Research Organization) grant OCENW.KLEIN.083 and
W.M.R. by the NWO (Dutch Research Organization) grants OCENW.KLEIN.083 and VI.Vidi.213.112.

\section*{Conflicts of interest}
The authors declare no conflicts of interest regarding this manuscript.

\section*{Data availability statement}
We do not analyse or generate any datasets, because our work proceeds within a theoretical and mathematical approach.

\section*{Acknowledgement}
The authors are grateful to A. Cipriani, R. Hazra and S. Janson for helpful discussions. F.C. would like to thank C. Bellingeri for pointing out the reference \cite{peccati_wiener_2011} in relation to Theorem \ref{thm:cumulants}.

\bibliographystyle{abbrv}
\bibliography{algebra, fkg, gff, xy, ktt, coulomb}

\end{document}